\documentclass[a4paper,12pt]{article}
\usepackage{amsthm,amsmath,mathrsfs,graphicx}
\usepackage{amsfonts,amssymb,xspace,ifpdf}
\newtheorem{theorem}{Theorem}
\newtheorem{lemma}[theorem]{Lemma}
\newtheorem{proposition}[theorem]{Proposition}
\newtheorem{conjecture}[theorem]{Conjecture}

\def\tc{total coloring\xspace}
\def\tcs{total colorings\xspace}
\def\cec{cyclically $\Delta$-edge-connected\xspace}
\def\soft#1{\leavevmode\setbox0=\hbox{h}\dimen7=\ht0\advance
    \dimen7 by-1ex\relax\if t#1\relax\rlap{\raise.6\dimen7
    \hbox{\kern.3ex\char'47}}#1\relax\else\if T#1\relax
    \rlap{\raise.5\dimen7\hbox{\kern1.3ex\char'47}}#1\relax
    \else\if d#1\relax\rlap{\raise.5\dimen7\hbox{\kern.9ex
    \char'47}}#1\relax\else\if D#1\relax\rlap{\raise.5\dimen7
    \hbox{\kern1.4ex\char'47}}#1\relax\else\if l#1\relax
    \rlap{\raise.5\dimen7\hbox{\kern.4ex\char'47}}#1\relax
    \else\if L#1\relax\rlap{\raise.5\dimen7\hbox{\kern.7ex
    \char'47}}#1\relax\else\message{accent \string\soft
    \space #1 not defined!}#1\relax\fi\fi\fi\fi\fi\fi}

\ifpdf
\fi
\begin{document}
\title{The last fraction of a fractional conjecture}
\author{Franti\v sek Kardo\v s\thanks{%
Institute of Mathematics, Faculty of Science, University of Pavol Jozef
\v{S}af\' arik, Jesenn\'a 5, 041 54 Ko\v{s}ice, Slovakia, and Institute
for Theoretical Computer Science, Faculty of Mathematics and Physics,
Charles University, Prague, Czech Republic. Partially supported by the
Slovak Science and Technology Assistance Agency under the contract No
APVV-0007-07. 
E-mail: \texttt{frantisek.kardos@upjs.sk}.}
\and
Daniel Kr{\'a}l'\thanks{%
Institute for Theoretical Computer Science, Faculty of Mathematics and
Physics, Charles University, Malostransk{\'e} n{\'a}m{\v e}st{\'\i} 25,
118 00 Prague, Czech Republic. E-mail: \texttt{kral@kam.mff.cuni.cz}.
The Institute for Theoretical Computer Science (ITI) is supported by
Ministry of Education of the Czech Republic	as project 1M0545.  This
research has been supported by grant GACR 201/09/0197.}
\and
Jean-S{\'e}bastien Sereni\thanks{%
CNRS (LIAFA, Universit\'e Denis Diderot), Paris, France, and Department
of Applied Mathematics (KAM), Faculty of Mathematics and Physics,
Charles University, Prague, Czech Republic.
E-mail: \texttt{sereni@kam.mff.cuni.cz}.}
}
\date{}	
\maketitle
\begin{abstract}
Reed conjectured that for every $\varepsilon>0$ and every integer
$\Delta$,
there exists $g$ such that the fractional total chromatic number of
every graph with maximum degree $\Delta$ and girth at least $g$ is at
most $\Delta+1+\varepsilon$. The conjecture was proven to be
true when $\Delta=3$ or $\Delta$ is even. We settle the conjecture by proving
it for the remaining cases.
\end{abstract}

\section{Introduction}
Fractional graph theory has led to many elegant and deep
results in the last three decades, broadening the range of applications
of graph theory and providing partial results and insights to many
hard problems. In this paper, we are interested in fractional
total colorings of graphs.
Total colorings form an extensively studied topic---see the monograph by
Yap~\cite{Yap96}---with the Total Coloring Conjecture of
Behzad~\cite{Beh65} and Vizing~\cite{Viz68} being one of its grails.
\begin{conjecture}[Total Coloring Conjecture]\label{conj-tot}
The total chromatic number of every graph with maximum degree $\Delta$ is
at most $\Delta+2$.
\end{conjecture}
The most important partial result toward the Total Coloring conjecture to date is
the following theorem proved by Molloy and Reed~\cite{MoRe98} in 1998.
\begin{theorem}
The total chromatic number of every graph with maximum degree $\Delta$ is
at most $\Delta+10^{26}$.
\end{theorem}
As is often the case, the fractional analogue of Conjecture~\ref{conj-tot} turns out
to be easier to approach. Kilakos and Reed~\cite{KiRe93}
proved in 1993 the following fractional analogue of the Total Coloring
Conjecture.
\begin{theorem}\label{thm.kire}
The fractional total coloring number of every graph with maximum degree
$\Delta$ is at most $\Delta+2$.
\end{theorem}
The graphs achieving the bound given in Theorem~\ref{thm.kire} have
been identified by Ito, Kennedy and Reed~\cite{IKR09}: they are the complete
graphs of even order and the complete bipartite graphs with equal part sizes.
Inspired by these results, Reed~\cite{Ree09} conjectured the following.
\begin{conjecture}\label{conj-reed}
For every $\varepsilon>0$ and every integer $\Delta$, there exists $g$
such that the fractional total chromatic number of every graph with
maximum degree $\Delta$ and girth at least $g$ is at most
$\Delta+1+\varepsilon$.
\end{conjecture}
Conjecture~\ref{conj-reed} was proven to be true when
$\Delta\in\{3\}\cup\{4,6,8,10,\ldots\}$ by Kaiser, Kr\'a\soft{l} and
King~\cite{KKK09} in the following stronger form.
\begin{theorem}\label{thm.kkk}
Let $\Delta\in\{3\}\cup\{4,6,8,10,\ldots\}$. There exists $g$ such that the fractional total chromatic
number of every graph with maximum degree $\Delta$ and girth at least $g$ is $\Delta+1$.
\end{theorem}
The purpose of our work is to settle Conjecture~\ref{conj-reed} by
proving that it also holds for odd values of $\Delta$. Our main theorem
reads as follows.
\begin{theorem}\label{thm-main}
For every $\varepsilon>0$ and every odd integer $\Delta\ge5$, there exists an integer $g$ such that
for the fractional total chromatic number of every graph with maximum degree $\Delta$ and girth at least $g$
is at most $\Delta+1+\varepsilon$.
\end{theorem}

The approach we use also yields a proof for the
case where $\Delta$ is even, and thus a full proof of
Conjecture~\ref{conj-reed}. However, we restrict ourselves to the case
of odd $\Delta\ge5$, since a stronger result,
Theorem~\ref{thm.kkk}, has been established~\cite{KKK09}.
Kaiser et al.~\cite{KKK09} also conjectured that the statement
of Theorem~\ref{thm.kkk} also holds for odd values of $\Delta$.
\begin{conjecture}
Let $\Delta\ge5$ be an odd integer.
There exists $g$ such that the fractional total chromatic
number of every graph with maximum degree $\Delta$ and girth at least $g$ is $\Delta+1$.
\end{conjecture}
However, we were not able to settle this stronger conjecture.

\section{Definitions and notation}
\label{sec.def}
Let us start by defining the relevant concepts.
For $X\subseteq\mathbb{R}$, we define $\mu(X)$ to be the Lebesgue measure of $X$.
If $G$ is a graph, then $V(G),E(G)$ and $\Delta(G)$ are its vertex-set,
edge-set and maximum degree, respectively. The \emph{total graph
$\mathcal{T}(G)$} of $G$ is the
graph with vertex-set $V(G)\cup E(G)$, where two vertices are
adjacent if and only if the corresponding elements of $G$ are adjacent
or incident. In other words, $xy\in E(\mathcal{T}(G))$ if and only if
\begin{itemize}
\item
$x,y\in V(G)$ and $xy\in E(G)$, or
\item
$x,y\in E(G)$ and $x$ and $y$ share a vertex in $G$, or
\item
$x\in V(G)$, $y\in E(G)$ and $y$ is incident to $x$ in $G$.
\end{itemize}
A \emph{total independent set} of $G$ is an independent set of
$\mathcal{T}(G)$.
Let $\Phi(G)$ be the set of all total independent sets of $G$.

Consider a function $w:\Phi(G)\to[0,1]$ and let $x\in V(\mathcal{T}(G))$.
We define $w[x]$ to be the sum of $w(I)$ over all $I\in\Phi(G)$
containing $x$. The mapping $w$ is a \emph{fractional $k$-total coloring} of
$G$ if and only if
\begin{itemize}
\item $\sum_{I\in\Phi(G)}w(I)\le k$; and
\item $w[x]\ge1$ for every $x\in V(\mathcal{T}(G))$.
\end{itemize}
Observe that $G$ has a fractional $k$-total coloring if and only if
there exists a function $c:V(\mathcal{T}(G))\to2^{[0,k]}$ such that
\begin{itemize}
\item $\mu(c(x))\ge1$ for every $x\in V(\mathcal{T}(G))$; and
\item $c(x)\cap c(y)=\emptyset$ for every edge $xy\in E(\mathcal{T}(G))$.
\end{itemize}
Notice that the second condition is the same as to require that $\mu(c(X)\cap
c(Y))=0$ whenever $x$ and $y$ are adjacent in $\mathcal{T}(G)$, since we consider
only finite graphs.
The \emph{fractional total chromatic number} of $G$ is the infimum of
all positive real numbers $k$ for which $G$ has a fractional $k$-total
coloring. As is well-known, the fractional total chromatic number of a
finite graph is always a rational number and the infimum is actually a
minimum.

A mapping $w:\Phi(G)\to[0,1]$ such that
\begin{itemize}
\item $\sum_{I\in\Phi(G)}w(I)=1$; and
\item $w[x]\ge\alpha$ for every $x\in V(\mathcal{T}(G))$
\end{itemize}
is a \emph{weighted $\alpha$-total coloring} of $G$. Observe that every fractional
$k$-total coloring yields a weighted $\frac{1}{k}$-total coloring $w$ such that
$w[x]\ge1/k$ for every $x\in V(\mathcal{T}(G))$. Conversely, one can derive
a fractional $\frac{1}{\alpha}$-total coloring from a weighted
$\alpha$-total coloring of $G$.
There are many equivalent definitions of a fractional coloring of a
graph, and we refer to the book by Scheinerman and Ullman~\cite{ScUl97}
for further exposition about fractional colorings (and, more generally,
fractional graph theory).

We now introduce some additional notation.
Two functions $f,g:X\to Y$ \emph{agree on $Z\subseteq X$} if the
restrictions of $f$ and $g$ to $Z$ are equal.
Let $G$ be a graph and $v\in V(G)$.
For a spanning subgraph $F$ of $G$, the degree of the vertex $v$ in $F$
is $\deg_F(v)$.
A spanning subgraph of $G$ with maximum degree at most $2$ is
a \emph{sub-$2$-factor} of $G$.
An \emph{$\ell$-decomposition} of $G$ is a partition of the edges
of $G$ into $\lceil\ell/2\rceil$ sub-$2$-factors, one of which is
required to be a matching if $\ell$ is odd.

Given a connected graph $G$,
an \emph{edge-cut} $F$ of $G$ is a subset of edges such that the removal
of $F$ disconnects $G$. Note that removing a minimal edge-cut splits $G$ into
exactly two connected components. An edge-cut $F$ is \emph{cyclic} if
every connected component of $G-F$ contains a cycle. A graph is
\emph{cyclically $k$-connected} if it has more than $k$ edges and no
cyclic edge-cut of size less than $k$.

We also use the following terminology from~\cite{KKSZ07}.
Let $H$ be a subgraph of a connected graph $G$. A path $P$ of
$G$ is an \emph{$H$-path} if both end-vertices of $P$ belong to $H$
but no internal vertex and no edge of $P$ belongs to $H$.
Given an integer $d$, the subgraph $H$ is \emph{$d$-closed} if the
length of every $H$-path is greater than $d$.
The \emph{$d$-connector} of $H$ in $G$ is the smallest $d$-closed
subgraph of $G$ that contains $H$.
The \emph{neighborhood $N(H)$ of $H$} is the subgraph of $G$ spanned by
all the edges of $G$ with at least one end-vertex in $H$.
We end this section by citing a lemma of Kaiser et al.~\cite[Lemma 8]{KKSZ07}
about connectors.

\begin{lemma}\label{lem.conn}
Let $d>\ell\ge1$. Suppose that $H$ is a subgraph of $G$ with at most
$\ell$ edges and no isolated vertices. If the girth of $G$ is greater
than $(d+1)\ell$, then the neighborhood of the $d$-connector of $H$ is a
forest.
\end{lemma}

\section{The cyclically $\Delta$-edge-connected case}
\label{sec.cyc}
We find a fractional $(\Delta+1+\varepsilon)$-total coloring of a given
\cec graph $G$ with odd maximum degree $\Delta\ge5$ in the following way: 
first, we decompose $G$ into a matching and a set of sub-$2$-factors, 
we then search for suitable weighted colorings corresponding to the factors 
and combine them into a weighted $\frac{1}{\Delta+1+\varepsilon}$-total coloring of $G$.

To find the decomposition, we use the following proposition~\cite[Proposition 1]{KKSZ07}.
\begin{proposition}\label{prop-kksz}
Every cyclically $\Delta$-edge-connected
graph with maximum degree $\Delta$ has a $\Delta$-decomposition.
\end{proposition}
For the next step, we need the following lemma.
\begin{lemma}\label{lem-mod}
Fix a positive real number $\varepsilon$ and a positive integer
$\Delta\ge 4$. There is an integer $g$ such that for every graph $G$
with maximum
degree $\Delta$ and girth at least $g$, and for every sub-$2$-factor $F$ of
$G$ such that $\Delta(G-F)\le\Delta-2$, there exists a weighted
\tc $w$ of $G$ such that
\begin{itemize}
\item $\forall v\in V(G),\quad w[v]\ge\frac{1}{\Delta+\varepsilon}$; and
\item $\forall f\in E(F),\quad
w[f]\ge\frac{\Delta-1}{2(\Delta+\varepsilon)}$.
\end{itemize}
\end{lemma}

Actually, Lemma~\ref{lem-mod} is implicit in~\cite{KKK09}. More
precisely, the following is proven.
\begin{lemma}[\protect{\cite[Lemma 18]{KKK09}}]\label{lem.18}
Let $\Delta\ge4$. For every $\varepsilon_0>0$, there exist $g\in\mathbb{N}$ and
$\alpha,\beta,\gamma \in \mathbb{R}^+$ satisfying 
$(\Delta-2)\alpha+\beta+2\gamma = 1$, $\alpha < \varepsilon_0$ such that
for every graph with maximum degree $\Delta$ and girth at least $g$, and
every $2$-factor $F$ of $G$, there exists a function $w:\Phi(G)\to[0,1]$
such that for every $x\in V(G)\cup E(G)$
\[
w[x]=\begin{cases}
  \beta&\text{if $x\in V(G)$,}\\
  \gamma&\text{if $x\in E(F)$,}\\
  \alpha&\text{otherwise.}
\end{cases}
\]
\end{lemma}
Actually, in Lemma 18, the parameter $\alpha$ is upper-bounded by
$4/3\ell$, with $\ell$ depending only on the girth of $G$; hence, it suffices to take $g$ sufficiently
large to obtain the inequality stated in Lemma~\ref{lem.18}.

As a crucial observation, Kaiser et al.~\cite[Proposition 19]{KKK09} proved
that in Lemma~\ref{lem.18} (Lemma 18 in~\cite{KKK09}),
the parameter $\beta$ can be chosen to be any value from the interval $(\varepsilon_0,Q(1))$
for a function $Q$ defined in~\cite{KKK09}. A more careful upper bound of one
of the parameters in the proof of their Proposition 19 allows us to
ensure that $\beta$ can be chosen to be $\frac1{\Delta+\varepsilon}$: it
suffices to prove that $Q(1)\ge\frac{1}{\Delta}$. For
$\Delta\in\{4,5,6\}$ it can be checked directly, and for $\Delta\ge 7$
we have
\[
Q(1)=\frac{1-F(1)^2}{2} \qquad\text{where} \qquad F(1) \le
1-\frac1{\Delta-2}\,.
\]
Since it holds that 
\[
1-\frac1{\Delta-2}\le \sqrt{1-\frac2\Delta}\,,
\]
the inequality $Q(1)\ge\frac{1}{\Delta}$ follows.

We are now ready to explain how to derive Lemma~\ref{lem-mod} from Lemma~\ref{lem.18}.
\begin{proof}[Proof of Lemma~\ref{lem-mod}]
Set $\varepsilon_0:=\frac{\varepsilon}{(\Delta-2)(\Delta+\varepsilon)}$,
and let $g$, $\alpha,\beta,\gamma$ be the constants given by
Lemma~\ref{lem.18}.  Let $G$ be a graph with maximum degree $\Delta$ and
girth at least $g$, and $F$ be a sub-$2$-factor of $G$.  We first build an
auxiliary graph $\widehat{G}$ as follows.

Set $k:=\sum_{v\in V(G)}(2-\deg_F(v))$. We view $G$ as the subgraph of
the multi-graph $G'$ obtained by adding to
$G$ a new vertex $v_0$ and, for each $v\in V(G)$, adding
$2-\deg_F(v)$ edges between $v_0$ and $v$. Thus, $\deg_{G'}(v_0)=k$ and
every other vertex of $G'$ has degree at most $\Delta$, since
$\Delta(G-F)\le\Delta-2$.
Let $H$ be a $k$-regular graph of girth at least $g$. The existence of
such a graph is well-known, consult, e.g., the book by Lov\'asz~\cite[Solution
of Problem 10.12]{Lov07}.
Replace every vertex $x$ of $H$ by a copy $G_x$ of $G$ and identify the
$k$ edges incident with $x$ in $H$ with the $k$ edges of $G'$ incident
with $v_0$.
Let $\widehat{G}$ be the obtained graph.
For $x\in V(H)$, let $F_x$ be the set of edges of $G_x$ corresponding to
the edges of $G$ that are in $F$.
Define $\widehat{F}$ to be the union of
$\bigcup_{x\in V(H)}F_x$ and the edges corresponding to those of $H$;
thus $\widehat{F}$ is a $2$-factor of $\widehat{G}$.

By the construction, $\widehat{G}$ has maximum degree $\Delta$
and girth at least $g$.
Thus, Lemma~\ref{lem.18} ensures the existence of a weighted total coloring
$w$ of $\widehat{G}$ such that $\alpha\le\varepsilon_0$, the weight of the
vertices is $\beta=\frac1{\Delta+\varepsilon}$ and the weight of the
edges in $\widehat{F}$ is
\[
\gamma = \frac12\left(1-(\Delta-2)\alpha-\beta\right)
\ge
\frac12\left(1-\frac{\varepsilon}{\Delta+\varepsilon}-\frac{1}{\Delta+\varepsilon}\right)
= \frac{\Delta-1}{2(\Delta+\varepsilon)}\,.
\]
This yields the conclusion since $\widehat{G}$ contains $G$ as a
subgraph and $\widehat{F}$ contains $F$.
\end{proof}
We conclude the section by proving Theorem~\ref{thm-main} restricted to
\cec graphs of maximum degree $\Delta$.
\begin{lemma}\label{prop.1}
Let $\Delta$ be an odd integer and $\varepsilon$ a positive real.
There exists $g$ such that the fractional total chromatic number of
every \cec graph with maximum degree $\Delta$ and girth at least $g$
is at most $\Delta+1+\varepsilon$.
\end{lemma}
\begin{proof}
We may assume that $\Delta\ge5$. Set $k:=\lfloor\Delta/2\rfloor$ and
$\varepsilon':=\varepsilon/2$.
Let $g$ be large enough so that Lemma~\ref{lem-mod} holds for the
fixed values of $\Delta$ and $\varepsilon'$.

By Proposition~\ref{prop-kksz}, the graph $G$ has a
$\Delta$-decomposition $M,F_1,F_2,\ldots,F_k$ where
$M$ is a matching.
In particular, $\Delta(G-F_i)\le\Delta-2$ for
$i\in\{1,\ldots,k\}$.
For every $i\in\{1,2,\ldots,k\}$, let $w_i'$ be
a weighted \tc of $G$ given by
Lemma~\ref{lem-mod} applied to $G$ and $F_i$ with respect to $\Delta$ and $\varepsilon'$. Further, let $w_0'$
be the weighted \tc of $G$ that assigns $1$ to the set $M$ and $0$ to every other
total independent set of $G$.
Finally, set $w_i:=\frac{2k+1}{2k(k+1)}\cdot w_i'$ for $i\in\{1,2\ldots,k\}$
and $w_0:=\frac{1}{2k+2}\cdot w_0'$.

A weighted $1/(\Delta+1+\varepsilon)$-coloring $w:\Phi(G)\to\mathbb{R}$
is defined by setting
$w:=\sum_{i=0}^{k}w_i$. Note that $w$ is a weighted \tc of $G$ since it
is a convex combination of weighted \tcs of $G$.

It remains to show that $w[x]$ is at least
$\frac{1}{\Delta+1+\varepsilon}=\frac{1}{2k+2+\varepsilon}$ for
every $x\in V(\mathcal{T}(G))$.
Let $v\in V(G)$. Then,
$w_i[v]\ge\frac{1}{\Delta+\varepsilon'}=\frac{2}{2(2k+1)+\varepsilon}$
for each $i\in\{1,2,\ldots,k\}$, and hence
\begin{align*}
w[v]&\ge\sum_{i=1}^{k}w_i[v]\\
&\ge k\cdot\frac{2k+1}{2k(k+1)}\cdot\frac{2}{2(2k+1)+\varepsilon}\\
&=\frac{1}{(k+1)\cdot\left(2+\frac{\varepsilon}{2k+1}\right)}\\
&>\frac{1}{2k+2+\varepsilon}\,.
\end{align*}
Now, let $e\in E(G)$. If $e\in M$, then
$w[e]\ge w_0(M)=\frac{1}{2k+2}$.
Otherwise,
there exists a unique $i\in\{1,2,\ldots,k\}$ such that $e\in E(F_i)$.
Then,
\[
w[e]\ge w_i[e]\ge\frac{2k+1}{2k(k+1)}\cdot\frac{2k}{2(2k+1+\varepsilon')}=\frac{1}{(k+1)\cdot\left(2+\frac{\varepsilon}{2k+1}\right)}>\frac{1}{2k+2+\varepsilon}\,.
\]
This concludes the proof.
\end{proof}

\section{The general case}
\label{sec.gen}
We start with an auxiliary lemma regarding recoloring of trees.
\begin{lemma}\label{lem.exttree}
Fix $\varepsilon>\varepsilon'>0$ and a positive integer $\Delta$.
There exists an integer $d$ such that the following holds.
\begin{itemize}
\item For every
tree $T$ rooted at a leaf $r$, with maximum degree $\Delta$ and depth
$d$,
\item for every fractional $(\Delta+1+\varepsilon')$-total coloring $c_0$ of $T$ 
with colors contained in
$[0,\Delta+1+\varepsilon']$, and
\item for every pair of disjoint sets
$X,Y\subset[0,\Delta+1+\varepsilon]$ each of measure $1$,
\end{itemize}
there exists a fractional $(\Delta+1+\varepsilon)$-total coloring $c$ of $T$ 
such that
\begin{enumerate}
\item $c(r)=X$ and $c(rr')=Y$ where $r'$ is the unique neighbor of $r$
in $T$, and
\item $c$ agrees with $c_0$ on all leaves and edges incident with leaves
that are at distance $d$ from $r$.
\end{enumerate}
\end{lemma}
\begin{proof}
Without loss of generality,
we may assume that $T$ is $\Delta$-regular, i.e., every vertex of $T$ has
degree either $1$ or $\Delta$, and all leaves are at distance $d$ from
$r$.
Set $\delta:=\varepsilon-\varepsilon'$,
$s:=\left\lceil\frac{\Delta+1+\varepsilon'}{\delta}\right\rceil$
and $d:=2s+1$.
We partition the vertices and edges of $T$ into levels regarding their
distance (in the total graph of $T$) from $r$. More precisely,
the root $r$ has level $d$, and
for every $i\ge0$, the level of a vertex at distance $i$ (in $T$) from
$r$ is $d-i$. Thus, the leaves of $T$ distinct from $r$ have level
$0$. The level of an edge that joins a vertex
of level $i$ with a vertex of level $i+1$ is $i$.

We fix an isometry $\pi$ of $[0,\Delta+1+\varepsilon]$ such that
$\pi(c_0(r))=X$ and $\pi(c_0(rr'))=Y$.
This is possible since $\mu(c_0(r))=\mu(c_0(rr'))=\mu(X)=\mu(Y)=1$ and
$c_0(r)\cap c_0(rr')=\emptyset=X\cap Y$.
Let $I_1,\ldots,I_s$ be a partition of $[0,\Delta+1+\varepsilon']$ into
sets of measure $\delta$, except $I_s$ which may have
a smaller measure; let $I_0=(\Delta+1+\varepsilon',\Delta+1+\varepsilon]$.
We set $K_k:=\pi(I_k)$ for $k\in\{1,\ldots,s\}$.

Let $\pi_0$ be the identity mapping of $[0,\Delta+1+\varepsilon]$.
We use a finite
induction to define a sequence $\pi_1,\ldots,\pi_{2d}$ of isometries of
$[0,\Delta+1+\varepsilon]$ such that for each $k$, the isometries
$\pi_{2k}$ and $\pi$ agree on the sets $I_1,\dots, I_k$.
The definition may be better digested with a look at Figure~\ref{fig.1}.

\begin{figure}
\centering
\includegraphics{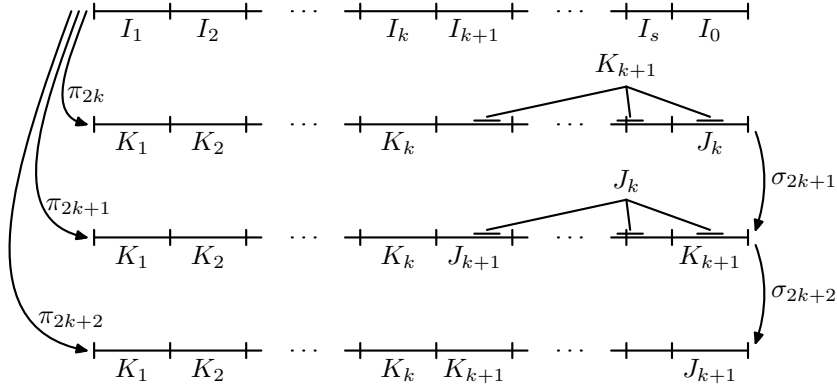}
\caption{An illustration of the actions of the isometries $\pi_{2k+1}$
and $\pi_{2k}$. The colors in $K_{k+1}$ are not used on level $2k+1$, and the
colors in $J_{k+1}$ are not used on level $2k+2$. Moreover, $\pi_{2k}$ and $\pi$
agree on $I_1\cup\ldots\cup I_k$.}
\label{fig.1}
\end{figure}

Let $k\in\{0,\dots,s-2\}$ and assume that
$\pi_{2k}$ is an isometry of $[0,\Delta+1+\varepsilon]$ such that
\[
\forall i\in\{1,\ldots,k\},\quad \forall t\in I_i,\quad
\pi_{2k}(t)=\pi(t)\,.
\]
Let $J_k:=\pi_{2k}(I_0)$, so $J_0=I_0$. 

In the odd step, we fix an isometry $\sigma_{2k+1}:J_k\to K_{k+1}$ such that
the restriction of $\sigma_{2k+1}$ to $J_k\cap K_{k+1}$ is the
identity mapping. The isometry $\pi_{2k+1}$ of
$[0,\Delta+1+\varepsilon]$ is defined by
\[
\pi_{2k+1}(t):=\begin{cases}
  \sigma_{2k+1}(\pi_{2k}(t))&\text{if $t\in\pi_{2k}^{-1}(J_k)=I_0$,}\\
  \sigma_{2k+1}^{-1}(\pi_{2k}(t))&\text{if $t\in\pi_{2k}^{-1}(K_{k+1})$,}\\
  \pi_{2k}(t)&\text{otherwise.}
\end{cases}
\]
For $i\in\{1,\ldots,k\}$, if $t\in I_i$ then
$t\notin\pi_{2k}^{-1}(K_{k+1})$ since $\pi_{2k}(t)=\pi(t)\in K_i$, and
$t\notin I_0$ since $I_0\cap I_i=\emptyset$.
Therefore, 
\[
\forall i\in\{1,\ldots,k\}, \quad \forall t\in I_i, \quad
\pi_{2k+1}(t)=\pi_{2k}(t)=\pi(t)\,.
\]
Note that $\pi_{2k+1}(I_0)=K_{k+1}$.
Let $J_{k+1}:=\pi_{2k+1}(I_{k+1})$. Since $\pi_{2k+1}$ is an isometry,
$J_{k+1}\cap K_i=\emptyset$ for all $i\in\{1,\dots,k+1\}$. 

In the even step, we first define the isometry
$\sigma_{2k+2}:J_{k+1}\to K_{k+1}$ by the condition
$\sigma_{2k+2}(\pi_{2k+1}(t))=\pi(t)$ for every $t\in I_{k+1}$. Next, the isometry $\pi_{2k+2}$ of $[0,\Delta+1+\varepsilon]$ is defined by
\[
\pi_{2k+2}(t):=\begin{cases}
  \sigma_{2k+2}(\pi_{2k+1}(t))&\text{if $t\in I_{k+1}$,}\\
  \sigma_{2k+2}^{-1}(\pi_{2k+1}(t))&\text{if $t\in I_0$,}\\
  \pi_{2k+1}(t)&\text{otherwise.}
\end{cases}
\]
It follows directly from the definition of $\pi_{2k+2}$ that
\[
\forall i\in\{1,\dots,k+1\},\quad \forall t\in I_i,\quad
\pi_{2k+2}(t)=\pi(t)\,.
\]
Notice that $\pi_{2k+2}(I_0)=J_{k+1}$.

When $k=s-1$, the set $K_{k+1}=K_s$ can have measure smaller than
$\delta$. We partition $J_{s-1}:=\pi_{2s-2}(I_0)$ into two sets $J_{s-1}^\prime$ and
$J_{s-1}^{\prime\prime}$ such that $\mu(J_{s-1}^\prime)=\mu(K_s)$. We
then continue in
the same manner and define an isometry $\pi_{2s-1}$ that
agrees with $\pi$ on all the sets $I_i$ with $i\in\{1,\ldots,s-1\}$ and
$\pi_{2s-1}(I_0)=K_s\cup J_{s-1}^{\prime\prime}$. Finally, we define an
isometry $\pi_{2s}$ that agrees with $\pi$ on
$[0,\Delta+1+\varepsilon']$.

Let us now define a coloring $c:V(T)\cup E(T)\to 2^{[0,\Delta+1+\varepsilon]}$ as follows:
\[
c(x):=\begin{cases}
  \pi_{i}(c_0(x))&\text{if $x$ has level $i\in\{0,\dots,2s\}$,}\\
  X&\text{if $x=r$.}
  \end{cases}
\]

Since $\pi_i$ is an isometry for $i\in\{0,\ldots,2s\}$ and $c_0$ is a fractional
total coloring of $T$, we have $\mu(c(x))\ge1$ for all $x\in V(T)\cup E(T)$.
To prove that $c$ is a fractional $(\Delta+1+\varepsilon)$-total
coloring of $T$ it suffices to prove that $c(x)\cap c(y)=\emptyset$ if $x$ and
$y$ are adjacent or incident in $T$. The levels of $x$ and $y$ can differ
by at most $1$. If the levels of $x$ and $y$ are the same, then
$c_0(x)\cap c_0(y)=\emptyset$ since $c_0$ is a fractional total coloring of
$T$. Hence,
\[
c(x)\cap c(y)=\pi_i(c_0(x))\cap\pi_i(c_0(y))=\emptyset\,.
\]
Let $x$ be a vertex or an edge of level $i$ and let $y$ be a vertex or
an edge adjacent to $x$ of level $i+1$, with $i\in\{0,\dots,2s-1\}$. Since
$c_0(x)\cap c_0(y)=\emptyset$, we have $\pi_i(c_0(x))\cap
\pi_i(c_0(y))=\emptyset$
as well. Since $c_0$ uses only colors from the interval
$[0,\Delta+1+\varepsilon']$, we also have $\pi_i(c_0(x))\cap
\pi_i(I_0)=\emptyset$.
The isometry $\sigma_i$ interchanges (some of the) colors from $\pi_i(I_0)$ (not
used in the level $i$) with some of the other colors, and hence,
$c(y)=\pi_{i+1}(c_0(y)) \subseteq \pi_i(c_0(y))\cup \pi_i(I_0)$.
Therefore,
\[
c(x)\cap c(y)\subseteq
\left(\pi_i(c_0(x))\cap\pi_i(c_0(y))\right)\cup\left(\pi_i(c_0(x))\cap\pi_i(I_0)\right)
=\emptyset\,.
\]

If $x$ has level $2s$ and $y$ has level $2s+1$, then $y=r$ and
$c(y)=X=\pi(c_0(r))$. On the other hand,
$c(x)=\pi_{2s}(c_0(x))=\pi(c_0(x))$. Therefore, the sets are disjoint.
To conclude, notice that
$c(x)=c_0(x)$ for all vertices and edges of level $0$
and that $c(rr')=\pi_{2s}(c_0(rr'))=\pi(c_0(rr'))=Y$.
\end{proof}

We are now ready to prove Theorem~\ref{thm-main}. 
\begin{proof}[Proof of Theorem~\ref{thm-main}]
Fix $\varepsilon'\in(0,\varepsilon)$. Let $d$ be
large enough so that Lemma~\ref{lem.exttree} holds for the values of
$\Delta$, $\varepsilon'$
and $\varepsilon$. Set $d_0:=2d+2$ and let $g$ be greater than
$(d_0+1)\cdot\Delta$ and such that Lemma~\ref{prop.1} holds for
$\Delta$ and $\varepsilon'$.

We proceed by induction on $|E(G)|$, the conclusion being trivial when
$|E(G)|\le\Delta$.
Now, if $G$ is \cec, then Lemma~\ref{prop.1} yields the result
(since $\varepsilon>\varepsilon'$).
So, we assume that $G$ is not \cec. Let $F$ be a (minimal) cyclic edge-cut
of $G$ such that $|F|<\Delta$ and $G-F$ is composed of two connected
components $A$ and $B$ such that $|B|$ is minimized.

Let $F'$ be the $d_0$-connector of $F$ in the subgraph $G_X$ of $G$ induced by
$B\cup F$. We now show that the subgraph $G_A$ of
$G$ induced by $A\cup N(F')$ has less edges than $G$.
By Lemma~\ref{lem.conn}, we know that $N(F')$ is a forest. On
the other hand, $B$ contains a cycle and hence $|E(N(F'))\setminus F|<|E(B)|$
(since $N(F')$ is contained in $G_X$ and the subgraph induced by $F$ is
acyclic by the girth requirement). Hence, $G_A$ has less edges than $G$,
maximum degree at most $\Delta$ and girth at least $g$. Therefore,
there exists a
fractional $(\Delta+1+\varepsilon)$-total coloring $c_A$ of $G_A$:
if
the maximum degree of $G_A$ is $\Delta$, then this follows from the induction
hypothesis, and otherwise it follows from Theorem~\ref{thm.kire}.

Let $G_B$ be the graph obtained from $G$ by contracting $A$ into a
single vertex $w$, and then subdividing $\lfloor g/2\rfloor$ times each edge incident with $w$;
thus the distance between $w$ and $B$ in $G_B$ is greater than $\lfloor
g/2\rfloor$.
Hence, the girth of $G_B$ is at least $g$.
Since $G_B$ contains more than $\Delta$ edges, $G_B$ is \cec, because any
cyclic edge-cut of $G_B$ yields a cyclic edge-cut of $G$ of at most the same
order, and whose removal splits $G$ into two components one of which is
smaller than $B$.
Consequently, Lemma~\ref{prop.1} ensures the existence of a
fractional $(\Delta+1+\varepsilon')$-total coloring $c_B$ of $G_B$.

Let $\mathscr{E}$ be the set of edges $xy$ of $G$ with
$x\in V(F')$ and $y\in V(B)\setminus V(F')$, i.e.,
$e$ is in $N(F')$ but not in $F'$.
For every $e=xy\in\mathscr{E}$ with $x\in V(F')$,
let $T_0(y)$ be the subgraph of $G_X-V(F')$
induced by the vertices at distance at most $d$ from $y$.
Let $T(e)$ be the graph obtained from $T_0(y)$ by adding $x$ and the
edge $e=xy$.
Observe that $T(e)$ is a tree, because the girth $g$ of $G$ is greater
than $2d+2$.
Moreover, if $e'=x'y'\in\mathscr{E}$ is distinct from $e$ (and $x'\in
V(F')$),
then $T(e)$ and $T(e')$ are vertex-disjoint unless $x=x'$ and then $x$
is the unique common vertex of $T(e)$ and $T(e')$, because $F'$ is
$d_0$-closed.

Now, for every edge $e=xy\in\mathscr{E}$, we apply
Lemma~\ref{lem.exttree} to the tree $T(e)$
with $c_0:=c_B$, $r:=x$, $X:=c_A(x)$ and $Y:=c_A(e)$. This yields a
fractional $(\Delta+1+\varepsilon)$-total coloring $c_e$ of $T(e)$,
which agrees with $c_B$ on all the leaves and
edges incident to a leaf that are at distance $d$ from $x$, if
any.

The property of disjointness of the trees ensures that $c_B$ along with
all the colorings $c_e$ for $e\in\mathscr{E}$ yield a fractional
$(\Delta+1+\varepsilon)$-total coloring $c'$ of the subgraph $B'$ of $G$
spanned by $B-F'$ and $\mathscr{E}$.

Since $G_A\cup B'=G$ and the colorings $c_A$ and $c'$ agree on all the
edges in $\mathscr{E}$ and all the vertices of $F'$ that are incident to
an edge in $\mathscr{E}$,
we obtain a fractional $(\Delta+1+\varepsilon)$-total coloring of $G$, as wanted.
\end{proof}

\bibliographystyle{plain}
\bibliography{kks09}

\begin{thebibliography}{10}

\bibitem{Beh65}
M.~Behzad.
\newblock {\em Graphs and their chromatic numbers}.
\newblock Doctoral thesis, Michigan State University, 1965.

\bibitem{IKR09}
T.~Ito, W.~S. Kennedy, and B.~Reed.
\newblock A characterization of graphs with fractional total chromatic number
  equal to {$\Delta+2$}.
\newblock In {\em Proceedings of the V Latin-American Algorithms, Graphs and
  Optimization Symposium (LAGOS 2009)}, Electronic Notes in Discrete Math.
\newblock To appear.

\bibitem{KKK09}
T.~Kaiser, A.~King, and D.~Kr{\'a}\soft{l}.
\newblock Fractional total colourings of graphs of high girth, 2009.
\newblock Manuscript arXiv:0911.2808.

\bibitem{KKSZ07}
T.~Kaiser, D.~Kr{\'a}{\soft{l}}, R.~{\v{S}}krekovski, and X.~Zhu.
\newblock The circular chromatic index of graphs of high girth.
\newblock {\em J. Combin. Theory Ser. B}, 97(1):1--13, 2007.

\bibitem{KiRe93}
K.~Kilakos and B.~Reed.
\newblock Fractionally colouring total graphs.
\newblock {\em Combinatorica}, 13(4):435--440, 1993.

\bibitem{Lov07}
L.~Lov{\'a}sz.
\newblock {\em Combinatorial problems and exercises}.
\newblock AMS Chelsea Publishing, Providence, RI, second edition, 2007.

\bibitem{MoRe98}
M.~Molloy and B.~Reed.
\newblock A bound on the total chromatic number.
\newblock {\em Combinatorica}, 18(2):241--280, 1998.

\bibitem{Ree09}
B.~Reed.
\newblock Fractional total colouring.
\newblock Presentation at the DIMACS Workshop on Graph Coloring and Structure,
  Princeton, May 2009.

\bibitem{ScUl97}
E.~R. Scheinerman and D.~H. Ullman.
\newblock {\em Fractional graph theory}.
\newblock Wiley-Interscience Series in Discrete Mathematics and Optimization.
  John Wiley \& Sons Inc., New York, 1997.

\bibitem{Viz68}
V.~G. Vizing.
\newblock Some unsolved problems in graph theory.
\newblock {\em Uspehi Mat. Nauk}, 23(6 (144)):117--134, 1968.

\bibitem{Yap96}
H.~P. Yap.
\newblock {\em Total colourings of graphs}, volume 1623 of {\em Lecture Notes
  in Mathematics}.
\newblock Springer-Verlag, Berlin, 1996.

\end{thebibliography}

\end{document}